\documentclass{amsart}

\usepackage{amsmath,amssymb, amsgen, amsthm, amsfonts, amscd, amsmath, bm, bbm, mathrsfs, enumerate, pdfsync, float, caption, subcaption, graphicx,  mathabx}

\newenvironment{psmallmatrix}
  {\left(\begin{smallmatrix}}
  {\end{smallmatrix}\right)}

\theoremstyle{plain}
\newtheorem{thm}{Theorem}
\newtheorem{lem}[thm]{Lemma}

\def\<{\langle} \def\>{\rangle}
\def\RR{\mathbb{R}}
\def\ZZ{\mathbb{Z}}  
\def\cM{\mathcal{M}}
\def\cQ{\mathcal{Q}}
\def\st{\;|\;}
\def\ra{\rightarrow}
\def\Sym{\text{Sym}}
\def\gl2r{\text{GL}(2, \RR)}
\def\sl2z{\text{SL}(2, \ZZ)}
\def\psl2z{\text{PSL}(2,\ZZ)}
\def\sltwor{\text{SL}(2,\RR)}
\def\b3z{B_3/Z(B_3)}
\def\zb3{Z(B_3)}
\def\symsym{\Sym(n)^3/\Sym(n)}

\title{An infinite family of non-cyclic 1-cylinder pillowcase-tiled surfaces}
\author[Abdalla]{Malak Abdalla}
\author[Brown]{Gabriela Brown}
\address{Department of Mathematics, University of Wisconsin-Madison, USA}
\email{gbrown29@wisc.edu}

\date{May 26, 2026}


\begin{document}

\begin{abstract}
Apisa-Wright conjectured that all branched covers of quadratic differentials in $\cQ(-1^4)$ with at most one cylinder in each direction are cyclic covers.
We provide infinitely many counterexamples to this conjecture.
\end{abstract}

\maketitle

\section{Introduction}
A \textit{(degree n) pillowcase-tiled surface} (PTS) is a (degree $n$) half-translation surface cover of a quadratic differential in $\cQ(-1^4)$ branched over, at most, the four poles. Such a surface is called \textit{1-cylinder} if the closure of every cylinder is the full surface. Apisa and Wright conjectured that all $1$-cylinder PTS are cyclic covers \cite[8.14]{ApisaWrightGeminal}. Our main result is a resolution of this conjecture with an infinite family of counterexamples.

\begin{thm}[Main result]\label{maintheorem} 
    When $n\geq 5$ is odd there exists a non-cyclic 1-cylinder degree $n$ PTS. No 1-cylinder degree $n$ PTS exists when $n$ is even.
\end{thm}

The Ornithorynque, discovered by Forni-Matheus \cite{ForniMatheus} (see also Forni-Matheus-Zorich \cite{FMZ}), was the first (holonomy cover of a) 1-cylinder PTS to be found. It is a degree $3$ cyclic cover and corresponds to the $(n,d) = (4,4)$ case in \cite[Table 10]{McBraidGroup}. Infinitely many more examples of (holonomy covers of) cyclic 1-cylinder PTS were found by Matheus and Yoccoz \cite{MatheusYoccoz} and Apisa and Wright \cite{ApisaWrightGeminal}. The famous Eierlegend-Wollmilchsau is not a holonomy cover of a 1-cylinder PTS
because there is no involution simultaneously exchanging every pair of homologous cylinders. This will follow from Theorem \ref{maintheorem}.

To construct our infinite family of counterexamples, we use the following purely algebraic characterization of being a 1-cylinder PTS. In what follows we let $\Sym(n)$ denote the symmetric group on $n$ elements. There is an action of $\Sym(n)$ on $\Sym(n)^k$ by simultaneous conjugation for any $k >0$. There is also an action of the braid group $B_3$ on $\Sym(n)^3$, which is recalled before Lemma \ref{braid_group}.

\begin{thm}\label{onecyl}
    A degree $n$ $1$-cylinder PTS is uniquely specified by a triple $(a,b,c)$ in $\Sym(n)^3/\Sym(n)$ satisfying the following condition: for every element $g$ of the braid group $B_3$, if $g \cdot(a,b,c) = (a',b',c')$, then $a'b'$ is an $n$-cycle.
\end{thm} 

The conjecture of Apisa and Wright developed from their work on a still-open conjecture of Gekhtman and Markovic. To state it, let $\cQ$ be a connected component of a stratum of quadratic differentials, and let $\cM_\cQ \subseteq \cQ$ be the subset of all quadratic differentials that generate a Teichm{\"u}ller disk that is a holomorphic retract of Teichm{\"u}ller space. In a stratum that has all even order zeros, Kra \cite{KraCaratheodory} showed that $\cM_\cQ = \cQ$ (see also McMullen \cite[Theorem 4.1]{Mc}). Gekhtman and Markovic conjecture \cite[Conjecture 1.1]{GekhtmanMarkovic} that $\cM_\cQ = \emptyset$ otherwise. 

By \cite[Lemma 2.2, \S 2.4, Corollary 4.4]{GekhtmanMarkovic} (see also the discussion in \cite[\S 1.4]{ApisaWrightGeminal}), $\cM_\cQ$ is an affine invariant subvariety (AIS) that is \textit{cylinder-free} i.e. closed under shearing and stretching any cylinder. In the case that $\cQ$ consists of squares of holomorphic 1-forms, Mirzakhani-Wright \cite{MirWri2} showed that the only cylinder-free AIS in $\cQ$ is $\cQ$. Classifying the cylinder-free AIS is a plausible approach to resolving the conjecture of Gekhtman-Markovic. Since Apisa-Wright \cite{ApisaWrightGeminal} showed the boundary of a cylinder-free AIS is still cylinder-free, one could use the inductive method of Apisa-Wright \cite{ApisaWrightHighRank} and the WYSIWYG boundary introduced by Mirzakhani-Wright \cite{MirWri} to classify all cylinder-free AIS. The base case of such an argument would be to classify all cylinder-free Teichm{\"u}ller curves. These are exactly the 1-cylinder PTS and 1-cylinder square tiled surfaces, the latter of which are flat tori by \cite{MirWri2}.

The ideas in our proof of Theorem \ref{onecyl} could lead to a new approach for classifying Shimura-Teichm{\"u}ller curves, which were classified in all but genus $5$ by Möller \cite{Moller}. The remaining case was classified with computer assistance by Aulicino-Norton \cite{AulicinoNorton}. The family we construct in Theorem \ref{maintheorem} also takes a step towards understanding square-tiled surfaces where all cusps have rank 1, as defined in \cite[Eq. 7.1]{McModularSymbols}, since these include all holonomy covers of 1-cylinder PTS.

\subsection*[]{Acknowledgements} We would like to thank Paul Apisa for his mentorship throughout this project, as well as the referees for their detailed and helpful feedback. This project started during the summer 2024 research program for undergraduates at UW Madison supported by NSF grant DMS 2304840, and continued in the fall 2024 Madison Experimental Mathematics lab supported by NSF DMS 2230900.

\section{Background}
The name pillowcase-tiled surface comes from a specific element $Y \in \cQ(-1^4)$ called the pillowcase, as shown in Figure \ref{pillowcase}. Fix a basepoint $y$, and let $Z$ denote the set of four branch points.
\begin{figure}[H]
    \centering
    \includegraphics[width=0.25\textwidth]{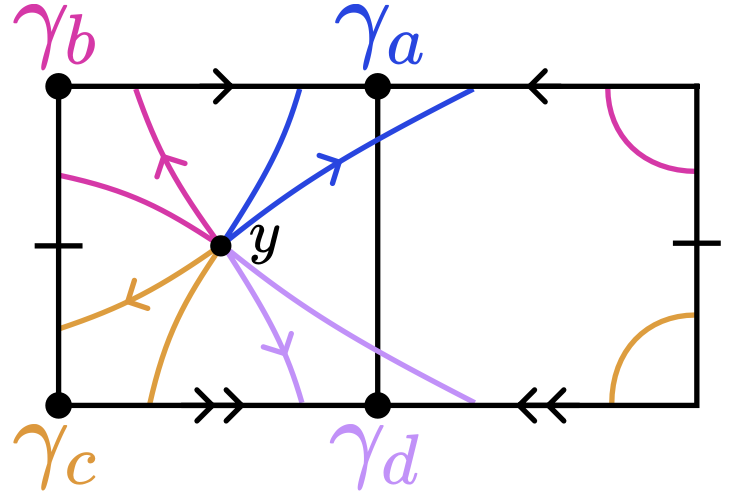} 
    \caption{The pillowcase $Y$, shown with generators of the fundamental group of the surface punctured at four marked points.}
    \label{pillowcase}
\end{figure}
It is sufficient to study covers of the pillowcase because every surface in $\cQ(-1^4)$ is in the $\sltwor$ orbit of the pillowcase. The stabilizer of the pillowcase under the $\sltwor$ action is $\sl2z$. We will use the following standard result. 

\begin{lem}\label{one-cyl}
    The pillowcase is a 1-cylinder surface and $\sl2z$ acts transitively on cylinder directions.
\end{lem}

To show whether or not a PTS is 1-cylinder it is better to work with the following equivalent algebraic characterization. 

\begin{lem}\label{PTS}
    The set of (possibly disconnected) degree $n$ PTS is in bijection with $\Sym(n)^3/\Sym(n)$ where $\Sym(n)$ acts by simultaneous conjugation. 
\end{lem}

\begin{proof}
    Let $p: X \ra Y$ be a PTS i.e. a cover of the pillowcase branched over the four points shown in Figure \ref{pillowcase}. Let $Z$ denote the set of branch points. The branched cover $p$ is uniquely determined by the cover $p': X-p^{-1}(Z) \ra Y-Z$. The fundamental group of the punctured pillowcase has presentation $\pi_1(Y-Z,y) = \<\gamma_a, \gamma_b, \gamma_c,\gamma_d \st \gamma_a\gamma_b\gamma_c\gamma_d = 1\>$ where the generators are the loops around each branch point as shown in Figure \ref{pillowcase}. The Galois correspondence associates $p'$ with a homomorphism $\rho: \pi_1(Y-Z,y) \ra \Sym(n)$ where $\rho(\gamma) = \sigma_\gamma$ is the permutation of the monodromy action of $\gamma$ on the lifts of the basepoint. Since $\rho$ is completely determined by the images of the generators, the associated PTS is determined by $(\rho(\gamma_a), \rho(\gamma_b), \rho(\gamma_c)) =: (a,b,c) \in \Sym(n)^3$, with the relation forcing $\rho(\gamma_d) = d = (abc)^{-1}$. The Galois correspondence is unique up to a change of basepoint, which is the same as a simultaneous conjugation of $(a,b,c)$ by some element $\sigma \in \Sym(n)$. 
\end{proof}

Whether or not a PTS has one cylinder in the horizontal direction can be checked in the following way.

\begin{lem}\label{n-cycle}
    A degree $n$ PTS $(a,b,c) \in \Sym(n)^3/\Sym(n)$ has exactly one cylinder in the horizontal direction if and only if $ab$ is an $n$-cycle.
\end{lem}

\begin{proof}
    Let $X$ be a PTS specified by $(a,b,c) \in \Sym(n)^3/\Sym(n)$. We may name the generators of $\pi_1(Y-Z,y)$ as in Figure \ref{pillowcase}, which makes $\gamma_a\gamma_b$ the core curve of the horizontal cylinder of $Y$. A horizontal cylinder on $X$ must project to the horizontal cylinder on $Y$. So the preimage in $X$ of $\gamma_a\gamma_b$ is the union of the core curves of the horizontal cylinders of $X$. As in Lemma \ref{PTS}, the permutation $ab$ describes the way the monodromy action of $\gamma_a\gamma_b$ permutes the fibers $p^{-1}(y)$. Since the lift of $\gamma_a\gamma_b$ is a union of core curves of horizontal cylinders, the permutation $ab$ is an $n$-cycle if and only if $X$ has one horizontal cylinder.
\end{proof}

Since the pillowcase has non-trivial translation automorphisms, we choose an action of $\sl2z$ that fixes the branch point encircled by $\gamma_d$ as in Figure \ref{pillowcase}. With this choice it turns out that the action of $\sl2z$ on the stratum restricts to a well defined action on the set of PTS's. A priori all we know is that the action of $A \in \sl2z$ fixes the pillowcase while sending a PTS $X$ to some other surface $X'$ in the same stratum. Letting $p: X \ra Y$ be the branched covering map, we can define a map $p' = ApA^{-1}$ from $X' \ra Y$. This is a conformal map, so $p'$ is in fact a translation surface covering map, and $X'$ is a PTS.

We now want to understand the action of $\sl2z$ on a PTS viewed as a triple in $\Sym(n)^3/\Sym(n)$. For this we will use the braid group on three strands, which has presentation $B_3 = \<\sigma_1, \sigma_2\st \sigma_1\sigma_2\sigma_1 = \sigma_2\sigma_1\sigma_2\>$ and acts on $\Sym(n)^3$ as follows:
    \begin{align*}
        \sigma_1 \cdot (a,b,c) &= (b,b^{-1}ab,c) \\  
        \sigma_2 \cdot (a,b,c) &= (a,c,c^{-1}bc).   
    \end{align*}
(See also \cite[Equation 4.8]{BirmanBrendle}). Let $Z(B_3)$ denote the center of $B_3$. Define $T := \begin{psmallmatrix}1 & 1 \\ 0 & 1\end{psmallmatrix}$ and $R := \begin{psmallmatrix}1 & 0 \\ -1 & 1\end{psmallmatrix}$, which together generate $\sl2z$.

\begin{lem}\label{braid_group}
    The isomorphism $\b3z \ra \psl2z$ by $\sigma_1 \mapsto T, \sigma_2 \mapsto R$ gives an isomorphism of the actions of these groups on the set of PTS. 
\end{lem}

\begin{proof} 
    Let $X$ be the PTS corresponding to $(a,b,c) \in \symsym$. As in Lemma \ref{PTS} let $\rho: \pi_1(Y{-}Z,y) \ra \Sym(n)$ be the defining homomorphism for $X$, where $Y$ is the pillowcase and $Z$ is the set of branch points. Note that changing the basepoint does not change the PTS because the image in $\Sym(n)$ only changes up to simultaneous conjugation. 
    Any $A \in \sl2z$ induces a homomorphism $\rho_A: \pi_1(Y{-}Z,y) \ra \pi_1(Y{-}Z,A(y))$ because $A$ preserves loops. The composition $\rho' = \rho\circ\rho_{A^{-1}}: \pi_1(Y{-}Z,y) \ra \Sym(n)$ is the homomorphism defining $A\cdot X$ as a PTS. So the action of $A$ on $X$ sends $(a,b,c) \mapsto (\rho'(\gamma_a), \rho'(\gamma_b), \rho'(\gamma_c))$.

    There is a surjection $\varphi: B_3 \ra \sl2z$ by $\sigma_1 \mapsto T$ and $\sigma_2 \mapsto R,$ which is a homomorphism because $TRT = \begin{psmallmatrix}0 & 1 \\ -1 & 0\end{psmallmatrix} = RTR$.  
    
    We will now describe the actions of $T$ and $R$ on $X$. Name the monodromy loops of $Y$ as in Figure \ref{pillowcase}, and let $a,b,c,d$ be the images of $\gamma_a,\gamma_b, \gamma_c, \gamma_d$ respectively under $\rho$. Figure \ref{rtactions} shows that the action of $T$ on $Y$ sends $\gamma_b \mapsto \gamma_a$, so $\rho'(\gamma_a) = \rho\circ\rho_{T^{-1}}(\gamma_a) = \rho(\gamma_b) = b$. Similarly $\rho'(\gamma_c) = c$ and $\rho'(\gamma_d) = d$. To see what $\rho'(\gamma_b)$ is, consider that $\gamma_a\gamma_b\gamma_c\gamma_d = 1$ in the fundamental group. Applying $\rho'$ and $\rho$ we get the relations $b\rho'(\gamma_b)cd = 1$ and $abcd = 1$ in $\Sym(n)$. Substituting $d = (abc)^{-1}$ into the first relation gives $\rho'(\gamma_b) = b^{-1}d^{-1}c^{-1} = b^{-1}(abc)c^{-1} = b^{-1}ab$. So $T\cdot (a,b,c) = (b,b^{-1}ab,c)$. A similar argument shows that $R\cdot(a,b,c) = (a,c,c^{-1}bc)$.

    \begin{figure}
        \centering
        \includegraphics[width=0.98\textwidth]{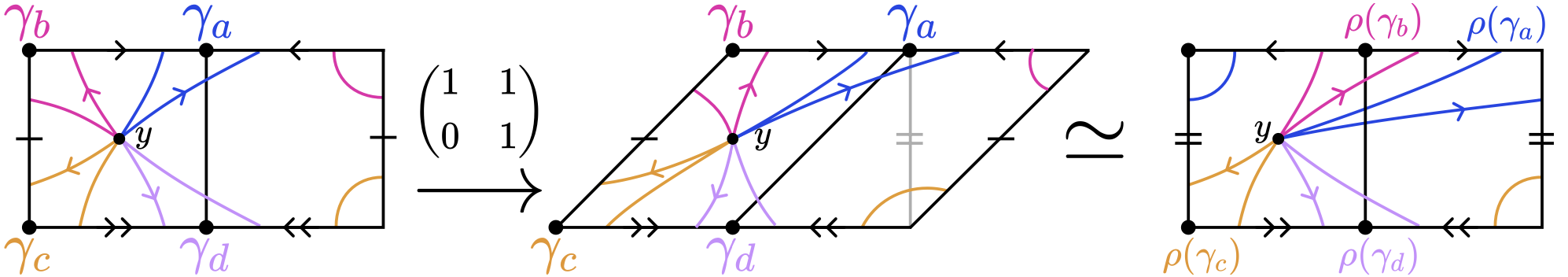} 
        \caption{The actions of $T$ on the pillowcase.}
        \label{rtactions}
    \end{figure}

    Since $T\cdot (a,b,c) = \sigma_1\cdot (a,b,c)$ and $R\cdot (a,b,c) = \sigma_2\cdot (a,b,c)$, and since $\varphi$ is a surjection from $B_3$ onto $\sl2z$, we have just shown that the action of $B_3$ on $\symsym$ factors through the action of $\sl2z$ on $\symsym$. 
    Note that both $T$ and $R$ fix the branch point that $\gamma_d$ encircles, so $-1 \in \sl2z$ acts trivially because it must also fix this branch point. So the action of $\sl2z$ in turn factors through an action of $\psl2z$. The map sending $\sigma_1 \ra T$ and $\sigma_2 \ra R$ induces a well-known isomorphism from $\b3z$ to $\psl2z$ that gives an isomorphism of their actions.
\end{proof}

Notably this implies that the $\sl2z$ orbit of a PTS is the same as the $B_3$ orbit of the corresponding element of $\symsym$.

\section{Proofs of main theorems}\label{maintheorems}
\begin{proof}[Proof of Thm \ref{onecyl}]
    Suppose $X$ is a PTS corresponding to $(a,b,c) \in \Sym(n)^3/\Sym(n)$. Since $\sl2z$ acts transitively on the cylinder directions of the pillowcase, it also acts transitively on the cylinder directions of any cover of the pillowcase because cylinder directions are preserved under the covering map. Fixing a cylinder direction on $X$ there is therefore some element of $\sl2z$ that sends that direction to the horizontal direction. By Lemma \ref{braid_group} this action is by some element of the braid group $g \in B_3$, and $(a',b',c') := g\cdot(a,b,c)$ corresponds to some other PTS $X'$. The fixed cylinder direction on $X$ has one cylinder exactly when $X'$ has one cylinder in the horizontal direction, which by Lemma \ref{n-cycle} happens if and only if $a'b'$ is an $n$-cycle.  Thus $X$ is a 1-cylinder surface if and only if for every triple $(a',b',c')$ in its $B_3$ orbit $a'b'$ is an $n$-cycle.
\end{proof}

\begin{lem}\label{counterexample}
    For any $n \geq 5$ odd, consider the permutations $a = (1\,2\,3\dots n)$ and $b = (1\, (n-1)\, 2\, (n-2) \dots (m+1)\, (m-1)\, m\, n)$ where $m = \frac{n+1}{2}$. The triple $(a,b,(ab)^{-1})$ defines a degree $n$ $1$-cylinder PTS with non-cyclic monodromy group $A_n$.
\end{lem}

\begin{proof}
    Let $a,b,c \in \Sym(n)$ be $n$-cycles such that $abc = 1$.
    The conjugacy classes of $\Sym(n)$ are by cycle type, so $b^{-1}ab$ and $c^{-1}bc$ are also $n$-cycles. Any element of the braid group $B_3$ therefore acts on $(a,b,c)$ by sending it to another triple $(a',b',c')$ of $n$-cycles with $a'b'c' = 1$. Then $a'b' = c'^{-1}$ which is an $n$-cycle. So $(a,b,c)$ satisfies Theorem \ref{onecyl} and thus specifies a $1$-cylinder PTS. Note that a similar proof works when we require one element of $(a,b,c)$ to be the identity and the other two to be $n$-cycles whose product is an $n$-cycle, but it will not lead to a different family of examples because it amounts to renaming the generating loops in Fig. \ref{pillowcase}.

    We will now show that such triples exist for all odd $n\geq 5$. As in the statement of the lemma, let $a= (1\,2\,3\dots n)$ and let $b = (1\, (n-1)\, 2\, (n-2) \dots (m+1)\, (m-1)\, m\, n)$ where $m= \frac{n+1}{2}$. Note that $ab$ is the $n$-cycle $(1 \,n \,2 \,(n-1) \dots (m-1)\, (m+1)\, m)$, so $c = (ab)^{-1}$ is also an $n$-cycle and $abc = 1$. Fix $n \geq 5$ so that $m\neq 2$. We now compute the commutator $[a,b^2]$, which depends on the parity of $m$. If $m$ is even, then $ab^2 = (1\, 3\, 5\,\dots\,(m-1))(2\,4\,\dots(m-2))$ and $(b^2a)^{-1} = (1\,(m-1)\,(m-3)\dots\,5\,3)(2\,n\,(m-2)\,(m-4)\dots 6\,4)$. If $m$ is odd, then $ab^2 = (1\,3\,5\dots (m-2)\, m \, 2\, 4 \dots (m-1))$ and $(b^2a)^{-1} = (1\,(m-1)\,(m-3)\dots 4\,2\,n\,(m-2)\,(m-4)\dots 5\,3)$. In both cases $[a,b^2]  = (2\,n\,m)$, and $\<a,b,c\>$ is non-abelian. 

    This example always generates the non-cyclic monodromy group $A_n$.  Recall that a \textit{block system} for a permutation group $G$ is a $G$-invariant partition of $\{1, \dots, n\}$. We say $G$ is \textit{primitive} if it only has the trivial block systems $\{\{1,2,\dots, n\}\}$ and $\{\{1\},\{2\}, \dots\{n\}\}$. Let $G = \<a,b,c\>$, and suppose $B$ is a nontrivial block system for $G$. Then $B$ is also a block system for the subgroup $\<a\>$. Since $a$ is an $n$-cycle, $B$ must be a partition into congruence classes modulo a divisor $d$ of $n$. If $G$ contains a three cycle, the three elements that are not fixed must either be in the same block or in three singleton blocks. But the latter case does not happen because congruence classes mod $d$ are all the same size, and $B$ is nontrivial. Since $(2\;n\;m) \in G$ the elements of $\{2,n,m\}$ must all be in the same block. This means they are in the same congruence class mod some divisor $d$ of $n$. So $2 \equiv 0 \;(\text{mod } d)$, but $n$ is odd, so $d=1$. This means $B$ was actually trivial and $G$ is primitive. A classical result of Jordan (see \cite[Thm 1.1]{Jones}) says if a primitive permutation group of degree $n$ contains a cycle of prime length fixing at least $3$ points, then the group must contain $A_n$. When $n \geq 7$, the cycle $(2\;n\;m)$ fixes more than $3$ points so $A_n \subseteq G$. But since $G$ is generated by $n$-cycles and $n$ is odd, $G \subseteq A_n$, so in fact $G = A_n$. 
    
    When $n=5$, $G$ is a transitive subgroup of $A_5$ generated by $5$-cycles, which means it's isomorphic to either $C_5$, $D_5$ or $A_5$. Since $[a,b^2] = (253)$, the order of $G$ is divisible by $3$, so $G = A_5$.
\end{proof}

\begin{proof}[Proof of Theorem \ref{maintheorem}]
    When $n \geq 5$ is odd, Lemma \ref{counterexample} constructs an infinite family of non-cyclic 1-cylinder PTS.

    When $n$ is even, we will show no 1-cylinder degree $n$ PTS exists. Let $t(\sigma)$ denote the number of disjoint cycles in a permutation $\sigma \in \Sym(n)$, including 1-cycles. Recall that $t(ab) = t(a) + t(b) - n\;(\text{mod } 2)$ for any $a,b \in \Sym(n)$. This formula holds because the sign of a permutation $\sigma$ is $(-1)^{n-t(\sigma)}$. Suppose $X$ is a 1-cylinder PTS corresponding to $(a,b,c) \in \Sym(n)^3/\Sym(n)$. Recall that the permutations $a,b,c$ correspond to the loops $\gamma_a,\gamma_b, \gamma_c$ on the pillowcase $Y$ as in Figure \ref{pillowcase}. The pairwise products $\gamma_a\gamma_b$, $\gamma_b\gamma_c$, and $\gamma_a\gamma_c$ are all core curves of cylinders on $Y$, and since $X$ is a 1-cylinder surface this means $ab$, $bc$, and $ac$ are all $n$-cycles. By the pigeonhole principle we can choose two distinct permutations $\sigma_1$ and $\sigma_2$ from $\{a,b,c\}$ such that $t(\sigma_1) = t(\sigma_2) \;(\text{mod }2)$, so $t(\sigma_1) + t(\sigma_2) -n = 0 \;(\text{mod } 2)$. But the product of any pair of them is an $n$-cycle so $t(\sigma_1\sigma_2) = 1 \;(\text{mod } 2)$, a contradiction.

\end{proof} 

Combinatorial results of Cangelmi \cite{Cangelmi} show that when $n\geq 5$ is odd there are many ways to choose three $n$-cycles whose product is the identity. This suggests that many groups could arise as the monodromy group of a 1-cylinder PTS. When $n$ is a prime $p$, a classical theorem of Jordan shows that any two non-commuting $p$-cycles generate $A_p$. Results of Jones \cite{Jones} have been used to  classify monodromy groups for primitive square-tiled surface with exactly one cylinder in the horizontal and vertical directions \cite[Thm 9.4]{AFBJM}. For  non-cyclic 1-cylinder PTS this classification implies that when $\<a,b,c\>$ is primitive and $n$ is large enough, the monodromy group is either $C_n, A_n$ or $\text{PGL}$.

\textbf{Question:} What groups can arise as the monodromy group of a 1-cylinder PTS?

\bibliography{bib}{}
\bibliographystyle{amsalpha}

\end{document}